\documentclass{amsart}
\usepackage{amscd,enumerate}
\usepackage{amssymb,amsmath,amsthm}
\usepackage{latexsym,amsfonts,amscd}
\usepackage{dsfont}

\usepackage[utf8]{inputenc}
\usepackage[OT4]{fontenc}

\newcommand*{\Q}{\ensuremath{\mathbb Q}}
\newcommand*{\N}{\ensuremath{\mathbb{N}}}

\newcommand*{\C}{\ensuremath{\mathbb{C}}}

\renewcommand*{\epsilon}{\varepsilon}

\newtheorem{theorem}{Theorem}
\newtheorem{definition}[theorem]{Definition}
\newtheorem{corollary}[theorem]{Corollary}
\newtheorem{lemma}[theorem]{Lemma}

\begin{document}
\title[Random variables with an invariant random shift]{Random variables with an invariant random shift in compact metrisable abelian groups}
\author{Michał Stanisław Wójcik}
\email{michal.s.wojcik@gmail.com}
\begin{abstract}
The main result of this paper states that for independent random variables $X, Y$ taking values in a compact metrisable abelian group, $X + Y$ has the same distribution as $X$, if and only if there exists a compact subgroup $A$ such that $P(Y\in A)=1$ and $X + a$ has the same distribution as $X$ for all $a\in A$. As a conclusion from the above it is shown that for independent random variables $X, Y$ such that $X+Y$ has the same distribution as $X$, $X+Y$ and $Y$ are also independent. It becomes also apparent that the distribution of $X$ is the Haar measure (uniform distribution) if for each open set $U$, $P(Y\in U) > 0$.
\end{abstract}
\maketitle

\subsection*{Important Note} After the first version of this paper was published on arxiv.org, I have learnt that my main result follows immediately from Choquet-Deny theorem. The condition $X + Y \sim X$ on a compact abelian group can be expressed as convolution equation of unknown Borel probability  measure $\nu$ and known Borel probability  measure $\mu$:
$$
\nu \ast \mu = \nu.
$$
For compact abelian group it follows from \cite{cd} that $\nu$ is periodic with periods in the subgroup generated by the support of $\mu$, from which immediately follows Theorem \ref{abelian_result}. The original result of G. Choquet and J. Deny was proven for locally compact abelian groups and in such case we require that $\nu$ must be bounded in some special sense (for details in English e.g. \cite[2.3]{Prio}). This ``boundary'' condition is always satisfied for each Borel probability measure on compact abelian groups.

\subsection*{Motivation}The direct inspiration for this paper was the latest research of Michał Ryszard Wójcik \cite{Woj}, where among other things he analyses independent real random variables $X$ and $Y$ such that $X + Y \mod 1 \sim X$. He obtains the results equivalent to Corollary \ref{mrw}. To show this he compiles some known results for characteristic functions of real random variables. From the theoretical point of view, the main inconvenience here is the fact, that even though we have $\varphi_{X + Y} = \varphi_X\cdot \varphi_Y$ for characteristics, we can't apply this smoothly to $X + Y \mod 1$. Luckily for $X,Y\in [0,1)$, we have if $\varphi_X(2\pi n)=\varphi_Y(2\pi n)$ for all $n\in\N$, then $X\sim Y$ \cite[XIX.6]{feller} and
for independent $X$ and $Y$, we have $\varphi_{(X+Y)\bmod{1}}(2\pi n)=\varphi_X(2\pi n)\varphi_Y(2\pi n)$ for all $n\in\N$ \cite{Woj}, which makes Wójcik's approach possible.\\
Since addition modulo $1$ on $[0,1)$ is equivalent to the multiplication on the circle group, we may expect that the methods of measure theory applied to topological groups will be adequate here. For locally compact second countable groups, the Haar measure is analogous to the uniform distribution. Since the theory of measure on locally compact abelian groups includes also characteristic functions (Fourier-Stillest transforms) of random variables taking values in the group (for details e.g \cite[1.3]{Rud}), we can use them explicitly, instead of the inconvenient, in this case, characteristics of real random variables. This makes the result a bit easier to achieve and more general.  

\subsection*{Result}
It is known that for random variables $X,Y$ taking values in a compact second countable group, if $X$ has the uniform distribution then $X+Y\sim X$ (e.g. \cite{Sta}). The main result of this paper states that for independent random variables $X$ and $Y$ taking values in a compact metrisable abelian group, $X + Y \sim X$ if and only if there exists a compact subgroup $A$, such that $P(Y\in A)=1$ and $X + a$ has the same distribution as $X$ for all $a\in A$. This result is given as Theorem \ref{abelian_result} and Theorem \ref{converse}. As a conclusion from the above, it is shown (Theorem \ref{independence}) that for independent $X, Y$ such that $X+Y\sim X$, $X+Y$ and $Y$ are also independent.

\
\

In this paper we will usually denote the group operation by '$+$' and identity element as 0. The only exception will be the circle group, where the group operation is complex multiplication. We will begin with recalling the basic definitions for the topological group and measures.

\begin{definition}
Let $G$ be a locally compact group. A complex function $\gamma:G\to \C$ is a character if and only if $\gamma$ is a continuous homomorphism ($\gamma(x + y) = \gamma(x)\cdot\gamma(y)$ for all $x,y\in G$) and $|\gamma(x)| = 1$ for all $x\in G$. The set of all characters will be called the dual group to $G$.
\end{definition}
    
\begin{definition}\label{init}
Let $G$ be a locally compact abelian group and $\Gamma$ be the dual group to $G$.
\begin{enumerate}
\item
$e(x) \overset{def}= 1$ for $x\in G$.
\item
$(x,y) \overset{def} = y(x)$ for all $x\in G$ and $y\in \Gamma$.
\item
$\mu \ast \nu$ is the convolution of measures $\mu$ and $\nu$.
$$
\mu \ast \nu(E) \overset{def}= \int_G\int_G \mathds{1}_E(x+y) d\mu(x)d\nu(y)
$$  
\item
Let $\mu$ will be a Borel probability regular measure on $G$. $\hat{\mu}$ a characteristic function of $\mu$, if and only if
$$
\hat{\mu}(y) \overset{def}= \int_G (x,\gamma) d\mu(x) \text{ for all } \gamma\in \Gamma.
$$ 
\end{enumerate}
\end{definition}

Let me cite the well know fact (e.g \cite[1.3.3]{Rud}).

\begin{theorem}\label{quick}
If $G$ is a locally compact abelian group and $\mu, \nu$ are Borel probability regular measures on $G$, then $\widehat{\mu\ast\nu} = \hat{\mu}\cdot\hat{\nu}$.
\end{theorem}

We will also recall uniqueness theorem (e.g. \cite[1.3.6]{Rud}).

\begin{theorem}\label{uniq}
If $G$ is a locally compact abelian group and $\mu, \nu$ are Borel probability regular measures on $G$ such that $\hat{\mu} = \hat{\nu}$, then $\mu = \nu$.
\end{theorem}

We will keep in mind the following folklore fact (e.g. \cite[412E, 412W]{Fre}).

\begin{theorem}
If $G$ is a compact second countable space, then each Borel probability measure on $G$ is regular.
\end{theorem}

\begin{definition}
Let $G$ be a topological space and $\Omega$ be a probabilistic space. 
\begin{enumerate}
\item
We will say that $X$ is a random variable if $X:\Omega\to G$ and $X^{-1}(B)$ is measurable for any Borel set $B$.
\item
Let $X$ be a random variable. $\mu_X(B) \overset{def}= P(X\in B)$. 
\end{enumerate}
\end{definition}

From Fubini theorem and Theorem \ref{quick}, we get immediately the following.

\begin{theorem}
If $G$ is a compact metrisable abelian group and $X,Y:\Omega\to G$ are independent random variables, then $\hat{\mu}_{X+Y} = \hat{\mu}_X\cdot\hat{\mu}_Y$.
\end{theorem}

\begin{lemma}\label{one}
If $G$ is a compact metrisable abelian group and $Y:\Omega\to G$ is a random variable, $\Gamma$ is dual to $G$, $\gamma\in \Gamma$ and $\hat{\mu}_Y(\gamma) = 1$, then $P((Y,\gamma) = 1) = 1$.
\end{lemma}
\begin{proof}
If $\int_G (x,\gamma) d\mu_Y = 1$, then $\int_G Re(x,\gamma) d\mu_Y = 1$. Thus $\int_G 1 - Re(x,\gamma) d\mu_Y$ = 0. Since $1 - Re(x,\gamma) \geq 0$ for each $x\in G$, $\mu_Y\{y\in G:Re(y,\gamma)=1\} = 1$. But since $|(y,\gamma)| = 1$ for all $y\in G$, we have $\mu_Y\{y\in G:(y,\gamma)=1\} = 1$. 
\end{proof}

Lemma \ref{one} is a simple generalisation of \cite[Proposition 7]{Woj}, which is a well known folklore fact. However, in our case we are not interested in values of $Y$, which in fact helps to simplify the reasoning. Indeed, the proof of the following theorem is a conceptual replacement for the line of reasoning in \cite[8 - 13]{Woj}.

\begin{theorem}\label{abelian_result}
If $G$ is a compact metrisable abelian group and $X,Y:\Omega\to G$ are independent random variables, such that $X\sim X+Y$, then there exists a compact subgroup $A\subset G$ such that $P(Y\in A) = 1$ and $X\sim X+a$ for each $a\in A$.
\end{theorem}
\begin{proof}
Let $\Gamma$ be the dual group to $G$. Note that
$$\hat{\mu}_X(\gamma) = \hat{\mu}_X(\gamma)\cdot \hat{\mu}_Y(\gamma) \text{ for all } \gamma\in \Gamma.$$
Let $\Lambda = \{\gamma\in\Gamma: \hat{\mu}_Y(\gamma) = 1\}$. Obviously $e\in \Lambda$, so $\Lambda\not=\emptyset$. Notice that $\hat{\mu}_X(\gamma) = 0$ for all $\gamma \in \Gamma\setminus \Lambda$. Let $A_\gamma = \{y\in G: (y,\gamma) = 1\}$. By Lemma \ref{one}, $\mu_Y(A_\gamma) = 1$ for all $\gamma\in\Lambda$. Since $G$ is compact and metrisable, $\Gamma$ is discrete and second countable (e.g. \cite[V.31]{Pon}). Hence $\Gamma$ has countable many elements. Let $A = \bigcap_{\gamma\in\Lambda}A_\gamma$. Since $A_\gamma$ is a compact subgroup for each $\gamma\in \Gamma$, then $A$ is also a compact subgroup and $\mu_Y(A)=1$.\\
Take any $a\in A$. Treat $a$ as a constant random variable. Thus $\hat{\mu}_a(\gamma) = (a,\gamma)$ for each $\gamma\in\Gamma$. Notice that
$$
\hat{\mu}_{X + a}(\gamma) = \hat{\mu}_X(\gamma)\cdot(a,\gamma) \text{ for each } \gamma\in\Gamma.
$$
For each $\gamma\in\Gamma\setminus \Lambda$, since $\hat{\mu}_Y(\lambda)\not=1$, we have $\hat{\mu}_X(\gamma) = 0$. Thus $\hat{\mu}_X(\gamma) = \hat{\mu}_X(\gamma)\cdot(a,\gamma)$ for all $\gamma\in\Gamma\setminus\Lambda$. On the other hand, for each $\gamma\in\Lambda$, since $(a,\gamma) = 1$, we have $\hat{\mu}_X(\gamma) = \hat{\mu}_X(\gamma)\cdot(a,\gamma)$. Thus generally $\hat{\mu}_{X + a}(\gamma) = \hat{\mu}_X(\gamma) \text{ for all } \gamma\in\Gamma$, so by Theorem \ref{uniq}, $X\sim X + a$.
\end{proof}

Since the Haar measure normalised to $1$ is unique on a compact group, we get immediately the following corollary. 
\begin{corollary}
If $G$ is a compact metrisable abelian group, $X,Y:\Omega\to G$ are independent random variables, such that $X\sim X+Y$ and $P(Y\in U)>0$ for each open $U\subset G$, then the distribution of $X$ is the Haar measure normed to $1$.
\end{corollary}

Corollary \ref{mrw} comes from \cite{Woj}. Equipped with Theorem \ref{abelian_result}, we can prove it keeping in mind some basic properties of the circle group and elementary facts of arithmetic. We will assume $0\in \N$. 

\begin{corollary}\label{mrw} 
Let $X,Y:\Omega\to [0,1)$ be independent random variables such that $X\sim X + Y \mod 1$. 
\begin{enumerate}
\item\label{finite_case}
If $0 < N = \min \Big\{n\in\N: P\big(Y \in \{\frac{k}{n}:k\in [0,n)\cap \N\}\big) = 1\Big\} < \infty$, then
$X \sim X + \frac{k}{N}$ for $k=0,\dots, N - 1$.
\item\label{infinite_case}
If $\{q\in \Q\cap[0,1):P(Y = q) > 0\}$ is infinite or $P(Y\in Q) < 1$, then $X\sim U[0,1)$. 
\end{enumerate}
\end{corollary}
\begin{proof}
Put $G = [0,1)$ with topology that glues $0$ and $1$. Addition modulo $1$ is a group operation on $G$. ($G$ is topologically isomorphic with the circle group). By Theorem \ref{abelian_result}, we have such a compact subgroup $A$, that $P(Y\in A) = 1$ and $X + a \mod 1 \sim X$ for all $a\in A$.
We will show (\ref{finite_case}). Since $N > 0$, we have coprime positive integers $p$ and $q$ such that $P(Y=\frac{p}{q}) > 0$. Thus $\frac{p}{q}\in A$. From Chinese reminder theorem, for each integer $r\in[0,q)$ there exist integers $k$ and $l$ such that $lq = kp - r$, hence $l + \frac{r}{q} = k\frac{p}{q}$. Thus $\{0, \frac{1}{q}, \dots, \frac{q-1}{q}\}\subset A$. So $q = N$ and  therefore $X + \frac{k}{N} \mod 1 \sim X$ for $k=1, \dots, N-1$.\\
To prove (\ref{infinite_case}) is enough to remember that any infinite subset of $G$ generates a dense subgroup and that each irrational number generates a dense subgroup.
\end{proof}
\begin{theorem}\label{independence}
If $G$ is a compact metrisable abelian group and $X,Y:\Omega\to G$ are independent random variables, such that $X\sim X+Y$, then $X + Y$ and $Y$ are independent.
\end{theorem}
\begin{proof}
By Theorem \ref{abelian_result}, we have a compact subgroup $A\subset G$ such that $P(Y\in A) = 1$ and $X\sim X+a$ for each $a\in A$.
Let $F:G^2\to G^2$ such that $F(x,y) = (x+y,y)$. Notice that $F$ is a homeomorphism. Then $F$ and $F^{-1}$ are both Borel.
Take any Borel set $B\subset G$ and any Borel set $E\subset A$.
\begin{equation}\label{first}
P(X + Y \in B \wedge Y\in E) = P((X,Y)\in F^{-1}(B\times E)) =
\end{equation}
$$
\int_G  \mathds{1}_{F^{-1}(B\times E)}(x,y) d\mu_X(x)d\mu_Y(y) = 
\int_G \int_G \mathds{1}_{(B\times E)}(x + y,y) d\mu_X(x)d\mu_Y(y) = 
$$
$$
\int_G \int_G \mathds{1}_B(x+y)\mathds{1}_E (y) d\mu_X(x)d\mu_Y(y) = 
\int_E \left(\int_G \mathds{1}_B(x+y) d\mu_X(x)\right) d\mu_Y(y) =
$$
\begin{equation}\label{sec}
\int_E P(X + y\in B) d\mu_Y(y) = \int_E P(X\in B) d\mu_Y(y) = P(X\in B)P(Y\in E) = 
\end{equation}
$$
P(X + Y\in B)P(Y\in E).
$$
Thus $X$ and $X + Y$ are independent random variables.
\end{proof}

\begin{corollary}
If $G$ is a compact metrisable abelian group and $X,Y:\Omega\to G$ are independent random variables, such that $X\sim X+Y$, then $X\sim X + nY$ for all $n\in\N$.
\end{corollary}

Now we will show the converse of  Theorem \ref{abelian_result}.
\begin{theorem}\label{converse}
If $G$ is a compact metrisable abelian group and $X,Y:\Omega\to G$ are independent random variables, $A\subset G$, $P(Y\in A) = 1$ and $X\sim X+a$ for each $a\in A$, then $X \sim X + Y$.
\end{theorem}
\begin{proof}
Put $E = A$. Take any Borel $B$. Obviously $P(X+Y\in B) = P(X+Y\in B \wedge Y\in E)$. Next we will follow the calculations from (\ref{first}) to (\ref{sec}). Hence $P(X+Y\in B) = P(X\in B)$.
\end{proof}

\end{document}